\documentclass[12pt,twoside,reqno]{amsart}
\usepackage{amsmath, amsfonts, amsthm, amssymb,amscd, graphicx, amscd}
\usepackage{float,epsf}
\usepackage[english]{babel}
\usepackage{enumerate}
\usepackage{tikz}

\usepackage{srcltx}
\usepackage{geometry}
\usepackage{verbatim}
\usepackage{mathrsfs}
\usepackage{hyperref}
\usepackage{enumitem} 
\usepackage{esint}

\newtheorem{theorem}{Theorem}[section]
\newtheorem{lemma}[theorem]{Lemma}

\newtheorem{proposition}[theorem]{Proposition}

\newtheorem{remark}[theorem]{Remark}
\newtheorem{example}[theorem]{Example}
\newtheorem{notation}[theorem]{Notation}

\numberwithin{equation}{section}
\numberwithin{figure}{section}

\newcommand{\norm}[1]{\left\|#1\right\|}

\def\intave#1{\int_{#1}\hbox{\llap{$\raise2.3pt\hbox{\vrule
height.9pt width7pt}\phantom{\scriptstyle{#1}}\mkern-2mu$}}}

\usepackage{graphicx} 

\title{Periodic Points of Diagonal and Permutation Operators}
\author{Howen Chuah}
\date{November 2024}

\begin{document}

\address{Howen Chuah,
Department of Mathematics, Purdue University,
 West Lafayette, IN 47907-2067, USA}
\email{hchuah@purdue.edu}

\keywords{Periodic Points, Diagonal Operators, Permutation Operators, Hilbert Spaces, Spectrum}
\subjclass[2020]{Primary: 46C05, 47A10, 47B02, 37B20, Secondary: 37C25}

\begin{abstract}
We first give a condition for a normal operator on a Hilbert space to have no nonzero periodic points, then we give a characterization of normal operators with the whole space as periodic points. We proceed to study the structure of periodic points of the diagonal operators and the permutation operators with examples. Moreover, it is also shown that the set of all diagonal operators with the whole space as periodic points is dense in the set of all unitary diagonal operators.
\end{abstract}

\maketitle

\section{Introduction}

The concept of periodic points and periods is of fundamental importance in the theory of dynamical systems. Recall that by a dynamical system we mean an ordered pair $(X,f),$ where $X$ is a metric space and $f$ is a continuous function which maps $X$ into itself. A point $p \in X$ is said to be a periodic point of $f$ if there exists a positive integer $m$ such that $f^m(p) = p,$ in which case we define the period (or the prime period) of $p$ to be the least positive integer $m$ satisfying $f^m(p) = p.$
 The general question is: Given a dynamical system $(X,f)$, can we characterize the periodic points of $f$
  and their periods? If so, can we say something interesting about their topological properties?

There have been plenty of literature answering the above question in different settings. For instance, \cite{{The_Set_of_Periods_of_Periodic_Points_of_a_Linear_Operator}} studied the case where $X$ is either $\mathbb{R}^n$ or $\mathbb{C}^n$ and $f$ is a linear operator, and it gave a characterization of the set of all periods of $f.$ In addition, \cite{{Periodic_Points_on_Hilbert_Space}} considered the case where $X$ is a Hilbert space and $f$ is a bounded linear operator. It showed that the periodic points are exactly countable unions of closed subspaces. Moreover, \cite{{Robert_Devaney}} studied continuous self maps on the real line, and \cite{{I_N_Baker}} studied polynomials on the complex plane.

In this paper, we continue to study the case $(H,T)$, where $H$ is a Hilbert space and $T$ is a bounded linear operator.
We study $T \in \mathcal{B}(H)$ and $P(T)$, where
\begin{equation}
\begin{array}{l}
\mathcal{B}(H): \mbox{$C^*$-algebra of all bounded linear operators on the Hilbert space $H$},\\
P(T): \mbox{the space of all periodic points of $T$.}
\end{array}
\label{dua}
\end{equation}
Clearly $P(T)$ is a $T$-invariant subspace of $H$, but it may not be closed in $H$.

For the rest of the paper, we shall always assume that the complex Hilbert space $H$ is infinite-dimensional and separable with countable orthonormal basis $\{e_n\}_{n \in \mathbb{N}}$, unless the contrary is stated.
Let $T \in \mathcal{B}(H)$. Recall that its spectrum is
\begin{equation}
 \sigma(T) := \{\lambda \in \mathbb{C}: T-\lambda I \text{ is not invertible in } \mathcal{B}(H)\}.
 \label{spe}
 \end{equation}
Let $G$ denote the roots of unity,
\begin{equation}
 G := \{e^{2\pi i q}: q \in \mathbb{Q}\} \subset \mathbb{C} .
 \label{rou}
 \end{equation}

Our first main result below shows that in order to study the periodic points of a normal operator $T$ (i.e. $T T^* = T^* T$), it suffices to consider $\sigma(T) \cap G$. Recall that by a well known result of Gelfand, for any $T \in \mathcal{B}(H)$,
we have $\sigma(T)$ is nonempty and compact. (See for instance \cite{{John_Conway}} Theorem 3.6 in p.196, or \cite{Gerard_J_Murphy} Theorem 1.2.5 in p.9.)

\begin{theorem}
\label{Condition_for_a_Normal_Operator_to_have_no_nonzero_periodic_points}
If $T \in \mathcal{B}(H)$ is normal and $\sigma(T) \cap G = \emptyset$, then $P(T) = \{0\}$.
Conversely, there exists a normal operator $S \in \mathcal{B}(H)$ such that $P(S) = \{0\}$ but $\sigma(S) \cap G \neq \emptyset$.
\end{theorem}

The next theorem studies normal operators whose periodic points are the whole space.
We show that every such operator must be diagonalizable and unitary.

\begin{theorem}
\label{structure_of_normal_operators_with_whole_space_as_periodic_points}
    Let $T \in \mathcal{B}(H)$ be a normal operator with $P(T) = H.$ Then there exists an orthonormal basis $\{u_n\}_{n \in \mathbb{N}}$ 
    such that $T(u_n) = c_n u_n$ for all $n \in \mathbb{N}$, and
    $\sigma(T) = \{c_n: n \in \mathbb{N}\} \subset G$ is finite.
     Thus $T$ is unitary.
    Moreover, if $N$ is the least common multiple of $\{\text{order}(c_n): n \in \mathbb{N}\},$ then $T^N = I.$
\end{theorem}

Now take $H_i := \mathbb{C}e_i$ for all $i \in \mathbb{N},$ (where $\{e_i\}_{i \in \mathbb{N}}$ is an orthonormal basis of $H$) so that the Hilbert space $H$ can be decomposed into the orthogonal direct sum of 1-dimensional subspaces as $H = \Sigma_{i \in \mathbb{N}} H_i.$ Suppose that $S \in \mathcal{B}(H)$ is a bounded linear operator that preserves the $H_i$s. (i.e. There is a bijection $\tau : \mathbb{N} \rightarrow \mathbb{N}$ with $S(H_i) = S(H_{\tau(i)})$ for all $i \in \mathbb{N}.$) Then the operator $S$ can be written as $S = S_1 S_2,$ where $S_1 \in \mathcal{B}(H)$ is a diagonal operator and $S_2 \in \mathcal{B}(H)$ is a permutation operator. This motivates us to study the structure of the periodic points of the two classes of normal operators, namely the diagonal ones and the permutation ones.
We fix an orthonormal basis $\{e_n\}_{n \in \mathbb{N}}$ of $H$.
In this way, we say that $T \in \mathcal{B}(H)$ is diagonal if there exists a bounded sequence $\{\alpha_n\}_{{n \in \mathbb{N}}} \subset \mathbb{C}$ such that $T(e_n) = \alpha_n e_n$ for all $n \in \mathbb{N}.$ It is easy to see that in this case, its operator norm is $\norm{T} = \sup_{n \in \mathbb{N}}|\alpha_n|.$ Moreover, the adjoint $T^*$ is also diagonal, and $T^*(e_n) = \bar \alpha_n e_n$ for all $n \in \mathbb{N}.$ Hence, $T$ is a normal operator.

Our third main theorem gives a characterization of the subspace of all periodic points of a diagonal operator. Recall from (\ref{rou})
that $G$ denotes the roots of unity.

\begin{theorem}
\label{Periodic_Points_of_a_Diagonal_Operator}
Let $T \in \mathcal{B}(H)$ be the diagonal operator given by $T(e_n) = \alpha_n e_n$ for all $n \in \mathbb{N}$,
where $\{\alpha_n\}_{n \in \mathbb{N}} \subset \mathbb{C}$ is bounded.
Then:
 \\(1) $P(T) = P(T^*) = \cup_{k \in \mathbb{N}} \{\sum_{n = 1}^{\infty} c_n e_n \in H: \{c_n\}_{n \in \mathbb{N}} \in l^2(\mathbb{N}), \alpha_n^k = 1 \text{ for all } c_n \neq 0\}$.
 \\(2) $P(T)$ is closed in $H$ if and only if $\{\alpha_n\}_{n \in \mathbb{N}} \cap G$ is finite.
 \\(3) $P(T) = H$ if and only if $\{\alpha_n\}_{n \in \mathbb{N}} \subset G$, and $\sup\{\text{order}(\alpha_n): n \in \mathbb{N}\} < \infty.$
 \\(4) If $\{\alpha_n\}_{n \in \mathbb{N}} \subset G$, and $\sup\{\text{order}(\alpha_n): n \in \mathbb{N}\} = \infty$, then $P(T)$ is a proper dense subspace of $H.$
\end{theorem}

We also study the permutation operators. Let $\sigma: \mathbb{N}\rightarrow \mathbb{N}$ be a bijection.
We define the permutation operator
\[
 T_{\sigma} \in \mathcal{B}(H) \;,\;
T_{\sigma}(e_n) = e_{\sigma(n)} \mbox{ for all } n \in \mathbb{N}.
\]
Clearly, $T_{\sigma}$ is invertible with inverse $T^*_{\sigma} =  T_{\sigma^{-1}}$ for each bijection $\sigma: \mathbb{N} \rightarrow \mathbb{N}.$ Therefore, $T_{\sigma}$ is a unitary operator, hence normal. Our fourth main theorem concerns the structure of periodic points of a permutation operator. The theorem parallels our third main theorem.

The bijection $\sigma$ induces an equivalence relation on $\mathbb{N}$, where $m,n \in \mathbb{N}$
are equivalent if and only if $m = \sigma^k(n)$ for some $k \in \mathbb{Z}$.
We denote the equivalence classes by
$[m] = \{\sigma^k(m): k \in \mathbb{Z}\}$,
and denote their cardinalities by $\text{card}([m])$.

\begin{theorem}
\label{Periodic_Points_of_a_Permutation_Operator}
 Let $\sigma: \mathbb{N} \rightarrow \mathbb{N}$ be a bijection, and let $T := T_\sigma$
 be its permutation operator.
 We have the followings:
 \\(1) $\cup_{k \in \mathbb{N}} \{\sum_{n = 1}^{\infty} c_n e_n \in H: \text{card}([n]) \leq k \text{ whenever } c_n \neq 0.\} \subset P(T)$. The inclusion can be proper, and $P(T) = P(T^*).$
 \\(2) $P(T)$ is closed in $H$ if and only if $\sup\{\text{card}([m]):m \in \mathbb{N}, [m] \text{ is finite.}\} < \infty.$
 \\(3) $P(T) = H$ if and only if each equivalence class is finite, and $\sup\{\text{card}([m]): m \in \mathbb{N}\} < \infty.$
 \\(4) If each equivalence class is finite and $\sup\{\text{card}([m]): m \in \mathbb{N}\} = \infty,$ then $P(T)$ is a proper dense subspace of $H.$
 \\(5) If the closure of $P(T)$ is fiinte-codimensional in $H,$ then $P(T)$ is dense in $H,$ so in fact the closure of $P(T)$ is the whole space $H.$
\end{theorem}

As a consequence of Theorem \ref{Periodic_Points_of_a_Diagonal_Operator}, we derive another interesting result. Our last result shows that every unitary diagonal operator $T \in \mathcal{B}(H)$ can be approximated by a sequence $\{T_n\}_{n \in \mathbb{N}}$ of diagonal operators with the property that $P(T_n) = H$.
In other words, the set of all diagonal operators $T$ with $P(T) = H$ is dense in the set of all unitary diagonal operators.

\begin{theorem}
\label{Apprximation_by_Diagonal_Operators_with_allspace_as_Fixed_Points}
Let $T \in \mathcal{B}(H)$ be an operator.
The following conditions are equivalent:
 \\(1) $T$ is unitary.
 \\(2) There exists a sequence $\{T_n\}_{n \in \mathbb{N}}$ of diagonal operators
 such that $P(T_n) = H$ for all $n \in \mathbb{N},$ and $\norm{T_n - T} \rightarrow 0$ as $n \rightarrow \infty.$
\end{theorem}

The main theorems are proved in the sections as follows:
\\Section 2: Theorem \ref{Condition_for_a_Normal_Operator_to_have_no_nonzero_periodic_points} and Theorem \ref{structure_of_normal_operators_with_whole_space_as_periodic_points}.
\\Section 3: Theorem \ref{Periodic_Points_of_a_Diagonal_Operator}.
\\Section 4: Theorem \ref{Periodic_Points_of_a_Permutation_Operator}.
\\Section 5: Theorem \ref{Apprximation_by_Diagonal_Operators_with_allspace_as_Fixed_Points}.

\section{Normal Operators with no Nonzero Periodic Points}

In this section, we prove Theorem \ref{Condition_for_a_Normal_Operator_to_have_no_nonzero_periodic_points}, which gives a condition for a normal operator to have no nonzero periodic point. The main tool is the continuous functional calculus in spectral theory, which we briefly review. Let
$H$ be a complex Hilbert space and let $T \in \mathcal{B}(H)$ be a normal operator.
Let $C^*\{I,T\}$ be the norm-closed $C^*$-subalgebra of $\mathcal{B}(H)$ generated by the identity operator $I$ and $T$.
As a result of the Gelfand transform, there exists a unique isometric $*$-isomorphism
\[ \phi : C(\sigma(T)) \rightarrow C^*\{I,T\}\]
 such that $\phi(p) = p(T)$ for any polynomial $p$.  Moreover, for two normal operators $T$ and $S$, $C(\sigma(T))$ and $C(\sigma(S))$ are $*$-isometric isomorphic as $C^*$-algebras if and only if the topological spaces $\sigma(T)$ and $\sigma(S)$ are homeomorphic (by the Banach-Stone Theorem, see for instance \cite{John_Conway} p.172).
This result classifies all commutative $C^*$-algebras up to isometric $*$-isomorphism.

\bigskip

{\bf Proof of Theorem \ref{Condition_for_a_Normal_Operator_to_have_no_nonzero_periodic_points}:}
 \\ Let $T$ be a normal operator.
 Recall from (\ref{dua}), (\ref{spe}) and (\ref{rou}) that
 $P(T)$ denotes the periodic points of $T$,
 $\sigma(T)$ denotes the spectrum of $T$, and $G$ denotes the roots of unity.
 We first show that
 \begin{equation}
 \sigma(T) \cap G = \emptyset \; \Longrightarrow \; P(T) = \{0\} .
 \label{shoo}
 \end{equation}

 Assume by contradiction that $\sigma(T) \cap G = \emptyset$, but there exists $0 \neq x \in P(T)$. Pick a $k \in \mathbb{N}$ such that $T^k(x) = x$. It then follows that $T^k-I$ is not invertible, so $1 \in \sigma(T^k)$. However, by the spectral mapping theorem (c.f. Theorem 2.1.14 in p.43 in \cite{Gerard_J_Murphy},) $\sigma(T^k) = \{z^k: z \in \sigma(T)\},$ so the above implies that there is a $z \in \sigma(T)$ which is also a $k$-th root of unity, we obtain a contradiction. This proves (\ref{shoo}).

 We proceed to show that there exists a normal operator $S \in \mathcal{B}(H)$ for which
 \begin{equation}
 \sigma(S) \cap G \neq \emptyset \;,\; P(S) = \{0\} .
 \label{shi}
 \end{equation}
  Let $\{q_n\}_{n \in \mathbb{N}}$ be an enumeration of the countable set $\mathbb{Q} \cap [0,1).$ Hence, each $t_n := q_n+\sqrt{2}$ is irrational and $\{e^{2\pi i t_n}: n \in \mathbb{N}\}$ is dense in the unit circle $S^1$. Now define a diagonal operator $S \in \mathcal{B}(H)$ with $S(e_n) = e^{2\pi i t_n} e_n$ for all $n \in \mathbb{N}.$ Note that $S^*$ is also diagonal with $S^*(e_n) = \bar e^{-2\pi i t_n} e_n$ for all $n \in \mathbb{N},$ so $S S^* = S^* S$ and so $S$ is normal. (In fact, $S$ is unitary.) Moreover, $\sigma(S)$ is the closure of the set $\{e^{2 \pi i t_n}: n \in \mathbb{N}\}$, which is the whole unit circle $S^1$.

  We now check that $P(S) =\{0\}$. Assume otherwise, namely there exists
  $0 \neq x = \sum_{n = 1}^{\infty} c_n e_n \in P(S)$. So there exists a $c_m \neq 0$. Pick $M \in \mathbb{N}$ for which $S^M(x) = x.$ Then $\sum_{n = 1}^{\infty} c_n e_n = x = S^M(x) = \sum_{n = 1}^{\infty} e^{2\pi i t_n M} c_n e_n$, and so
  \[ c_m = \langle \sum_{n = 1}^{\infty} c_n e_n,e_m \rangle  =
  \langle \sum_{n = 1}^{\infty} e^{2\pi i t_n M} c_n e_n,e_m \rangle = e^{2\pi i t_m M} c_m .\]
  Since $c_m \neq 0$, this implies that $e^{2\pi i t_m M} = 1$, which contradicts the fact that $t_m$ is irrational.
  This proves (\ref{shi}). The theorem follows from (\ref{shoo}) and (\ref{shi}).
$\Box$

\begin{remark}
{\rm (1) Consider the case where the Hilbert space is the finite dimensional $\mathbb{C}^n$. Let $T: \mathbb{C}^n \rightarrow \mathbb{C}^n$ be a normal operator. In this case, $\sigma(T)$ is simply the set of all eigenvalues of $T,$ which is a finite set in $\mathbb{C}$. Theorem \ref{Condition_for_a_Normal_Operator_to_have_no_nonzero_periodic_points} says that if $\sigma(T)$ does not contain a root of unity, then $T$ has no periodic points other than $0$.
\\(2) One can come up with other examples satisfying the conditions of part (2) of  Theorem \ref{Condition_for_a_Normal_Operator_to_have_no_nonzero_periodic_points}. Here we give another such example. Take $H$ to be $l^2(\mathbb{Z}),$ and for each $k \in \mathbb{Z}$ take $\delta_k \in l^2(\mathbb{Z})$ to be sequence given by $\delta_k(n) = 1$ if $n = k$ and $\delta_k(n) = 0$ otherwise. Consider the bilateral shift operator $S:l^2(\mathbb{Z}) \rightarrow l^2(\mathbb{Z})$ given by $S(\delta_k) = \delta_{k+1}$ for all $k \in \mathbb{Z}$.
Then $S$ is a unitary, and hence normal operator. Also, $\sigma(S)$ is the unit circle, and $P(S) =\{0\}$.
\\(3) Note that Theorem \ref{Condition_for_a_Normal_Operator_to_have_no_nonzero_periodic_points} only gives a necessary condition
for the existence of $0 \neq x \in P(T)$, but not a sufficient condition.
Therefore, as a possible direction for future studies, it is natural to ask if one could have other interesting necessary or sufficient conditions for the existence of $0 \neq x \in P(T)$.}
$\Box$
\end{remark}

We proceed to study normal operators $T$ where $P(T)=H$. We shall show that they are unitary, diagonalizable, and of finite exponent
(i.e. $T^k = I$ for some $k \in \mathbb{N}$).
We shall need the following basic lemma in the proof of Theorem \ref{structure_of_normal_operators_with_whole_space_as_periodic_points}.

\begin{lemma}
\label{orthogonality_lemma}
Let $M$ and $N$ be both closed subspaces of the Hilbert space $H.$ Assume that $M$ is orthogonal to $N$.
Then the sum $M+N := \{x+y: x \in M, y \in N\}$ is also closed in $H.$
\end{lemma}
\begin{proof}
    Let $\{z_n\}_{n \in \mathbb{N}} \subset M+N$ be a sequence and $z_n \rightarrow z$ for some $z \in H$ as $n \rightarrow \infty.$ We will show that $z \in M+N$ also. Write each $z_n = x_n+y_n,$ where $x_n \in M$ and $y_n \in N.$ Note that $\norm{x_n-x_m} \leq \norm{z_n-z_m} \rightarrow 0$ as $n,m \rightarrow \infty.$ (by the fact that $\{z_n\}_{n \in \mathbb{N}}$ is a Cauchy sequence.) This shows that $\{x_n\}_{n \in \mathbb{N}}$ is a Cauchy sequence in $M.$ By the completeness of $M,$ $x_n \rightarrow x$ for some $x \in M.$ Similarly, $y_n \rightarrow y$ for some $y \in N.$ It follows that $z = x+y \in M+N.$ We complete the proof.
\end{proof}

We are now ready to give a proof of Theorem \ref{structure_of_normal_operators_with_whole_space_as_periodic_points}.
Let $Per(T)$ denote set of all periods of the operator $T$.

\bigskip

{\bf Proof of Theorem \ref{structure_of_normal_operators_with_whole_space_as_periodic_points}:}


We have $P(T) = H$, which is clearly closed in itself, so by Theorem 3.3 in \cite{{Periodic_Points_on_Hilbert_Space}}, it follows that $Per(T)$ is a finite set. Let $M$ be the least common multiple of elements of $Per(T)$. Then
\begin{equation}
T^M = I .
\label{onee}
\end{equation}

Consider the polynomial $q(z) = z^M-1$. Let $\phi: C(\sigma(T)) \rightarrow C^*\{I,T\}$ be the continuous functional calculus for $T$. We see that $\phi(q) = q(T) = T^M-I = 0.$ Since $\phi$ is an isometric $*$-isomorphism, it follows that $q$ vanishes on $\sigma(T)$ identically. Hence, each point of $\sigma(T)$ is a root of $q(z)$. This implies that $\sigma(T)$ is finite. Thus we can write
\[ \sigma(T) = \{\lambda_1, ..., \lambda_m\} .\]
 Each $\lambda_i$ is isolated, so it is an eigenvalue of $T$,
     and we let $E_i := \{x \in H: Tx = \lambda_i x\}$ be its eigenspace. Note that $E_i = ker(T-\lambda_i I),$ so $E_i$ is closed in $H.$ Moreover, note that since $T$ is normal, $\{E_i\}_1^m$ are mutually orthogonal. (See \cite{Friedberg_Insel_Spence} p.371 Theorem 6.15 for instance.) By applying Lemma \ref{orthogonality_lemma} iteratively, and we see that the sum $E_1+...+E_m$ is also closed in $H$.

     We claim that
     \begin{equation}
     E_1+...+E_m = H.
      \label{jum}
      \end{equation}
      To see this, we assume otherwise and pick a nonzero $0 \neq v \in H$ which is orthogonal to $E_1+...+E_m$. Let $V := \{p(T)v: p \text{ is a complex polynomial.}\},$ which is the smallest $T-$invariant subspace of $H$ generated by $v.$ By (\ref{onee}), $T^M = I,$ so $V$ is finite-dimensional. Moreover, $V$ is orthogonal to $E_1+...+E_m.$ The restriction $T|_V : V \rightarrow V$ has an eigenvalue $\lambda$ and a corresponding eigenvector $u$. Note that $\lambda \in \sigma(T) = \{\lambda_1, ..., \lambda_m\},$ so $\lambda = \lambda_j$ for some $j.$ However this implies that $u \in E_j,$ which contradicts to the fact that $E_j \cap V = \{0\}$. This proves (\ref{jum})
      as claimed.

      Take orthonormal basis $\mathcal{B}_i$ for each eigenspace $E_i$. By (\ref{jum}),  the union $\cup_{i = 1}^m \mathcal{B}_i$ is an orthonormal basis for $H$ consisting entirely of eigenvectors of $T$. It follows that
$T$ is diagonalizable, namely there is an orthonormal basis $\{u_n\}_{n \in \mathbb{N}}$ of $H$ such that
\[ T(u_i) = c_i u_i \mbox{ for all } i \in \mathbb{N} .\]
Recall that by (\ref{onee}), $T^M = I,$ so for each $i,$ $u_i = T^M(u_i) = c_i^M u_i.$ So each $c_i$ is an $M$-th root of unity. Consequently, $|c_i| = 1$ for all $i,$ and this shows that $T$ is unitary. Moreover, $\sup\{\text{order}(c_n): n \in \mathbb{N}\} \leq M < \infty.$ Let $N$ to be the least common multiple of $\{\text{order}(c_i): i \in \mathbb{N}\},$ and one has $T^N = I.$ We complete the proof.
$\Box$

We give a few remarks regarding Theorem \ref{structure_of_normal_operators_with_whole_space_as_periodic_points}.

\begin{remark}
    {\rm Note that if $K$ is just an infinite-dimensional complex inner product space, which is not assumed to be a Hilbert space, then it is not true that a bounded linear operator $T: K \rightarrow K$ with $P(T) = K$ must satisfy $T^N = I$ for some $N \in \mathbb{N}.$ We give an example for this. Again take $H$ to be an infinite-dimensional separable complex Hilbert space with countable orthonormal basis $\{e_n\}_{n \in \mathbb{N}}.$ Take $K = \cup_{n \in \mathbb{N}} \{c_1 e_1+...+c_n e_n: c_1, ..., c_n \in \mathbb{C}.\}.$ (That is , $K$ is the space of all finite linear combinations of $\{e_n\}_{n \in \mathbb{N}},$ and $\{e_n\}_{n \in \mathbb{N}}$ is a Hamel basis of $K.$) Now define $T: K \rightarrow K$ by $T(e_n) = \xi_n e_n$ for all $n \in \mathbb{N},$ where $\xi_n := e^{\frac{2 \pi i}{n}}.$ Then every element of $K$ is periodic under $T,$ but there is no $N \in \mathbb{N}$ for which $T^N = I.$}
$\Box$
\end{remark}
The above example shows the importance of completeness in Theorem \ref{structure_of_normal_operators_with_whole_space_as_periodic_points}.

The next example shows that the normality of the operator also plays an important role in Theorem \ref{structure_of_normal_operators_with_whole_space_as_periodic_points}.

\begin{example}
    {\rm Consider the operator $T \in \mathcal{B}(H)$ defined by $T(e_1) = 2e_2,$ $T(e_2) = \frac{1}{2} e_1,$ and $T(e_j) = e_j$ for all $j \geq 3,$ $j \in \mathbb{N}.$ It is then clear that $P(T) = H$ and that $T$ is not a normal operator. In this case, $T$ is not unitary as asserted in Theorem \ref{structure_of_normal_operators_with_whole_space_as_periodic_points}. This shows that the normality of $T$ is crucial in Theorem \ref{structure_of_normal_operators_with_whole_space_as_periodic_points}.}
$\Box$
\end{example}

\section{Periodic Points of a Diagonal Operator}

By Theorem \ref{structure_of_normal_operators_with_whole_space_as_periodic_points}, a normal operator with the whole space as periodic points must be diagonalizable. It is therefore instructive to study the diagonal operators in more detail. We shall give a more detailed treatment of such operators here. In this section, we prove Theorem \ref{Periodic_Points_of_a_Diagonal_Operator} and examine some basic examples. We first characterize the space of all periodic points of a diagonal operator. In addition, we also show that for a diagonal operator $T \in \mathcal{B}(H),$ the space of all periodic points remains the same under taking adjoints.

\begin{proposition}
\label{Characterization_of_Periodic_Points}
    Let $\{\alpha_n\}_{n \in \mathbb{N}}$ be a bounded sequence of complex numbers, and let $T \in \mathcal{B}(H)$ be a diagonal operator with $T(e_n) = \alpha_n e_n$ for all $n \in \mathbb{N}.$ Then $P(T) = P(T^*) = \cup_{k \in \mathbb{N}} \{\sum_{n = 1}^{\infty} c_n e_n \in H: \{c_n\}_{n \in \mathbb{N}} \in l^2(\mathbb{N}), \alpha_n^k = 1 \text{ whenever } c_n \neq 0\}.$
\end{proposition}
\begin{proof}
Firstly, we show that
\begin{equation}
\label{characterization_of_P(T)_diagonal}
P(T) = \cup_{k \in \mathbb{N}} \{\sum_{n = 1}^{\infty} c_n e_n \in H: \{c_n\}_{n \in \mathbb{N}} \in l^2(\mathbb{N}), \alpha_n^k = 1 \text{ whenever } c_n \neq 0\}.
\end{equation}
 Assume $x = \Sigma_{n = 1}^{\infty} c_n e_n \in H, \text{ where }\{c_n\}_{n \in \mathbb{N}} \in l^2(\mathbb{N})$ and that there is a $k \in \mathbb{N}$ such that $\alpha_n^k = 1 \text{ whenever } c_n \neq 0.$ Then
 \[ T^k(x) = T^k(\Sigma_{n = 1}^{\infty} c_n e_n) = T^k(\Sigma_{c_n \neq 0} c_n e_n) = \Sigma_{c_n \neq 0} \alpha_n^k c_n e_n = \Sigma_{c_n \neq 0} c_n e_n = \Sigma_{n = 1}^{\infty} c_n e_n = x.
  \]
  This shows that $x \in P(T).$ Conversely, suppose that $x \in P(T),$ so there exists a $k \in \mathbb{N}$ with $T^k(x) = x.$ Assume that $x = \Sigma_{n = 1}^{\infty} c_n e_n \in H, \text{ where }\{c_n\}_{n \in \mathbb{N}} \in l^2(\mathbb{N}).$ Then the above implies that $\Sigma_{n = 1}^{\infty} c_n e_n = x = T^k(x) = \Sigma_{n = 1}^{\infty} \alpha_n^k c_n e_n.$ It follows that $\alpha_n^k = 1$ whenever $c_n \neq 0.$ The above argument proves \eqref{characterization_of_P(T)_diagonal}.

It remains to show that
$P(T) = P(T^*)$.
Since $T$ is a diagonal operator with $T(e_n) = \alpha_n e_n$ for all $n \in \mathbb{N},$ it follows that $T^*$ is also a diagonal operator with $T^*(e_n) = \bar \alpha_n e_n$ for all $n \in \mathbb{N}.$ So
\[
\begin{array}{rl}
P(T) & = \cup_{k \in \mathbb{N}} \{\sum_{n = 1}^{\infty} c_n e_n \in H: \{c_n\}_{n \in \mathbb{N}} \in l^2(\mathbb{N}), \alpha_n^k = 1 \text{ whenever } c_n \neq 0\} \\
& = \cup_{k \in \mathbb{N}} \{\sum_{n = 1}^{\infty} c_n e_n \in H: \{c_n\}_{n \in \mathbb{N}} \in l^2(\mathbb{N}), \bar \alpha_n^k = 1 \text{ whenever } c_n \neq 0\} \\
& = P(T^*).
\end{array}
\]
We complete the proof.
\end{proof}

In other words, an element $x = \Sigma_{n = 1}^{\infty} c_n e_n \in H$ is a periodic point of the diagonal operator $T$ if and only if there exists a positive integer $k \in \mathbb{N}$ such that any $\alpha_n$ with $c_n \neq 0$ is a $k$-th root of unity. We now apply Proposition \ref{Characterization_of_Periodic_Points} to consider some examples.

\begin{example}
{\rm We study an example arising from Fourier analysis.
Let $H := L^2_{\mathbb{C}}(-\pi,\pi)$, the square-integrable complex-valued functions on $(-\pi,\pi)$.
Let  $u_n := \frac{1}{\sqrt{2\pi}} e^{inx} \in H$ for all $n \in \mathbb{Z}.$
It is well-known that $\{u_n\}_{n \in \mathbb{Z}}$ forms an orthonormal basis of $H.$ Fix $h \in H$, and let $\lambda_n := \frac{1}{\sqrt{2\pi}}\int_{-\pi}^{\pi} h(x) e^{-inx} dx$ for all $n \in \mathbb{Z}$ be the $n$-th Fourier coefficient. Define $K \in \mathcal{B}(H)$ by convolution:
\begin{equation}
    Kf := \frac{1}{\sqrt{2\pi}}\int_{-\pi}^{\pi} h(x-y)f(y) \, dy \mbox{ for all } f \in H = L^2_{\mathbb{C}}(-\pi,\pi).
\end{equation}
Then
\[
\begin{array}{rl}
    (K u_n)(x)
    & = \frac{1}{\sqrt{2\pi}}\int_{-\pi}^{\pi} h(x-y) \frac{1}{\sqrt{2\pi}} e^{iny} \,dy \\
    & = \frac{1}{2\pi}\int_{-\pi}^{\pi} h(x-y) e^{iny} \,dy \\
   & = \frac{1}{2\pi}\int_{x-\pi}^{x+\pi} h(u) e^{in(x-u)} \,du \; (\text{let } u = x-y.) \\
   & = \frac{1}{2\pi}\int_{-\pi}^{\pi} h(u) e^{in(x-u)} \, du \\
    & = \lambda_n u_n.
\end{array}
\]
The above computation shows that $K$ is diagonal with respect to $\{u_n\}_{n \in \mathbb{Z}}.$ Now, by the familiar Riemann-Lebesgue lemma in Fourier analysis (See for instance \cite{{Wheeden_Zygmund}} Theorem 12.21 in p.312.), $\lambda_n \rightarrow 0$ as $n \rightarrow \infty.$ Hence, there exists only finitely many $\lambda_n$ that are roots of unity. By Proposition \ref{Characterization_of_Periodic_Points}, it follows that the space $P(K)$ of all periodic points of the convolution operator $K$ is finite-dimensional.}
$\Box$
\end{example}

We consider another interesting example.

\begin{example}
{\rm In this example, we use Proposition \ref{Characterization_of_Periodic_Points} to construct a unitary diagonal operator $T \in \mathcal{B}(H)$ such that $P(T)$ is an increasing union of countably many finite-dimensional subspaces of $H,$ and the closure of $P(T)$ has codimension $1$ in $H$. Let $\alpha_1 = e^{2\sqrt{2}\pi i}$, and let $\alpha_n := e^{\frac{2\pi i}{2^{n-1}}}$
for all $n \geq 2$. Let $T \in \mathcal{B}(H)$ be the diagonal operator with $T(e_n) = \alpha_n e_n$ for all $n.$ Note that $T^*(e_n) = \bar \alpha_n e_n = \alpha_n^{-1} e_n = T^{-1}(e_n),$ so that $T^* = T^{-1}$ and $T$ is unitary.
Now, in order to find the space $P(T),$ we note that for each $k \in \mathbb{N},$ we write $k$ as $k = 2^{k_1}k_2,$ where $k_1$ is a non-negative integer and $k_2 \in \mathbb{N}$ is odd. The set $\{\sum_{n = 1}^{\infty} c_n e_n \in H: \{c_n\}_{n \in \mathbb{N}} \in l^2(\mathbb{N}), \alpha_n^k = 1 \text{ whenever } c_n \neq 0\}$ is the linear span of $\{e_n: 2 \leq n \leq k_1\},$ (and it is the zero subspace when $k_1 = 0$). Therefore, by Proposition \ref{Characterization_of_Periodic_Points}, $P(T)$ is the
direct sum of the linear spans of $\{e_n\}_{n \geq 2},$ so the closure of $P(T)$ has codimension 1.}
$\Box$
\end{example}

Note that by using a similar construction as in the above example, for each $m \in \mathbb{N}$ one could construct a unitary operator $T \in \mathcal{B}(H)$ such that $P(T)$ is a union of countably many finite-dimensional subspaces, and the closure of $P(T)$ has codimension $m.$ (Take $\alpha_1 = \alpha_2 = ... = \alpha_m = e^{2\sqrt{2}\pi i},$ and for all $n > m,$ define $\alpha_n := e^{\frac{2\pi i}{2^n}}.$)

We proceed to give a necessary and sufficient condition for the space of all periodic points of a diagonal operator to be closed in $H.$
Recall from (\ref{rou}) that $G$ is the set of all roots of unity.

\begin{proposition}
\label{Criterion_for_the_Periodic_Points_to_be_Closed}
      Let $\{\alpha_n\}_{n \in \mathbb{N}} \subset \mathbb{C}$ be a bounded sequence, and let $T \in \mathcal{B}(H)$ be a diagonal operator with $T(e_n) = \alpha_n e_n$ for all $n \in \mathbb{N}.$ Then $P(T)$ is closed in $H$ if and only if $\{\alpha_n: n \in \mathbb{N}\} \cap G$ is finite.
\end{proposition}
\begin{proof}
We first assume that $\{\alpha_n: n \in \mathbb{N}\} \cap G$ is infinite. Then there exists an increasing sequence
$\{n_k\}_{k \in \mathbb{N}} \subset \mathbb{N}$ such that
\[ \alpha_{n_j} \in G \mbox{ for all } j \in \mathbb{N} \;,\;
\mbox{order}(\alpha_{n_1}) < \mbox{order} (\alpha_{n_2}) < \mbox{order}(\alpha_{n_3}) < ... .\]
Hence order$(\alpha_{n_j}) \rightarrow \infty$ as $j \rightarrow \infty$.
For each $k \in \mathbb{N},$ we define
\[ x_k := \sum_{j = 1}^{k} \frac{1}{2^j}e_{n_j} = \frac{1}{2}e_{n_1}+\frac{1}{2^2}e_{n_2}+...+\frac{1}{2^k}e_{n_k} \in H. \]
Since $\alpha_{n_1}, ..., \alpha_{n_k} \in G,$ there is a positive integer $N$ for which $\alpha_{n_j}^N = 1$ for all $j = 1, ..., k.$ It follows that
\[ T^N(x_k) = \sum_{j = 1}^{k} \frac{1}{2^j}T^N(e_{n_j}) = \sum_{j = 1}^{k} \frac{1}{2^j}\alpha_{n_j}^Ne_{n_j} = \sum_{j = 1}^{k} \frac{1}{2^j}e_{n_j} = x_k. \]
 Consequently, $x_k \in P(T)$. However, observe that $x_k \rightarrow x := \sum_{j = 1}^{\infty} \frac{1}{2^j}e_{n_j},$ and
 we claim that
\begin{equation}
x \notin P(T).
\label{buka}
\end{equation}
Assume otherwise that $x \in P(T)$,
so there is a positive integer $N$ for which $T^N(x) = x.$ Then we have
\[ \sum_{j = 1}^{\infty} \frac{1}{2^j}e_{n_j} = x = T^N(x) = T^N(\sum_{j = 1}^{\infty} \frac{1}{2^j}e_{n_j}) = \sum_{j = 1}^{\infty} \frac{1}{2^j}T^N(e_{n_j}) = \sum_{j = 1}^{\infty} \frac{1}{2^j} \alpha_{n_j}^N e_{n_j},\]
which implies that $\alpha_{n_j}^N = 1$ for all $j \in \mathbb{N}$,
 so order$(\alpha_{n_j})$ $\leq N$ for all $j \in \mathbb{N}.$ This contradicts to the fact that order$(\alpha_{n_j}) \rightarrow \infty$ as $j \rightarrow \infty.$ This proves (\ref{buka}) as claimed.
 We have shown that if $\{\alpha_n: n \in \mathbb{N}\} \cap G$ is infinite, then $P(T)$ is not closed in $H$.

Conversely, assume that the set $\{\alpha_n: n \in \mathbb{N}\} \cap G$ is finite. Let $k \in \mathbb{N}$ be the smallest positive integer  such that $\alpha_n^k = 1$ whenever $\alpha_n \in G.$ Let $I$ be the identity operator, and
we claim that
\begin{equation}
P(T) = \text{ker}(T^k-I).
\label{clam}
\end{equation}
 Clearly $\text{ker}(T^k-I) \subset P(T)$. To prove the opposite inclusion, let $x = \sum_{n = 1}^{\infty} c_n e_n \in P(T)$, where $\{c_n\}_{n \in \mathbb{N}} \in l^2(\mathbb{N})$.
  Thus, $T^m(x) = x$ for some $m \in \mathbb{N}.$ We then have that $\sum_{n = 1}^{\infty} c_n e_n = x = T^m(x) = \sum_{n = 1}^{\infty} c_n \alpha_n^k e_n,$ so $c_n \alpha_n^m = c_n$ for all $n \in \mathbb{N}.$ Now if $n$ is such that $c_n \neq 0,$ from the above we would have that $\alpha_n \in G.$ Hence $x = \sum_{\alpha_n \in G} c_n e_n$ and $T^k(x) = \sum_{\alpha_n \in G} c_n \alpha_n^k e_n = \sum_{\alpha_n \in G} c_n e_n = x$, namely $x \in \text{ker}(T^k-I).$
  This proves (\ref{clam}) as claimed.
  It implies that $P(T)$ is closed in $H.$ We complete the proof.
\end{proof}

As a corollary of Proposition \ref{Criterion_for_the_Periodic_Points_to_be_Closed}, we derive the following criterion for determining whether $P(T)$ is the whole space $H$ or is dense in $H.$

\begin{proposition}
\label{Criterion_for_the_Periodic_Points_to_be_the_whole_or_dense}
   Let $\{\alpha_n\}_{n \in \mathbb{N}} \subset \mathbb{C}$ be a bounded sequence, and let $T \in \mathcal{B}(H)$ be a diagonal operator with $T(e_n) = \alpha_n e_n$ for all $n \in \mathbb{N}.$
    \\(1) $P(T) = H$ if and only if $\{\alpha_n\}_{n \in \mathbb{N}} \subset G$, and $\sup\{\text{order}(\alpha_n): n \in \mathbb{N}\} < \infty.$
    \\(2) If $\{\alpha_n\}_{n \in \mathbb{N}} \subset G$, and $\sup\{\text{order}(\alpha_n): n \in \mathbb{N}\} = \infty,$ then $P(T)$ is a proper dense subspace of $H.$
\end{proposition}
\begin{proof}
(1) Firstly, assume that $P(T) = H.$ Then for each $n \in \mathbb{N},$ there is a positive integer $k \in \mathbb{N}$ such that $T^k(e_n) = e_n.$ Therefore, $e_n = T^k(e_n) = \alpha_n^k e_n,$ from which we obtain $\alpha_n^k = 1$ and so $\alpha_n \in G$. Now note that $P(T) = H,$ which is a closed subspace of $H.$ Therefore, by Proposition \ref{Criterion_for_the_Periodic_Points_to_be_Closed}, $\{\alpha_n\}_{n \in \mathbb{N}} \cap G$ is finite. This implies that $\sup\{\text{order}(\alpha_n): n \in \mathbb{N}\} < \infty.$

  Conversely, assume that $\{\alpha_n\}_{n \in \mathbb{N}} \subset G$,
   and $\sup\{\text{order}(\alpha_n): n \in \mathbb{N}\} < \infty.$ In this case, one can pick a $k \in \mathbb{N}$ such that $\alpha_n^k = 1$ for all $n \in \mathbb{N}.$ It follows that $T^k$ is the identity operator, and so $P(T) = H.$
\\(2) Assume that $\{\alpha_n\}_{n \in \mathbb{N}} \subset G$,
and $\sup\{\text{order}(\alpha_n): n \in \mathbb{N}\} = \infty.$
For each $n \in \mathbb{N},$ there is a $k \in \mathbb{N}$ such that $\alpha_n^k = 1.$ Thus $T^k(e_n) = \alpha_n^k e_n = e_n,$ so $e_n \in P(T).$
Hence $P(T)$ contains the linear span of $\{e_n\}_{n \in \mathbb{N}},$ which is $\cup_{n \in \mathbb{N}} \{c_1 e_1+...+c_n e_n: c_1, ..., c_n \in \mathbb{C}\}$ and is dense in $H.$ This shows that $P(T)$ is dense in $H.$ However, observe that by (1), we have that $P(T) \neq H.$ Consequently, $P(T)$ is a proper dense subspace of $H.$ We complete the proof.
\end{proof}

By combining the above propositions, we can now prove Theorem \ref{Periodic_Points_of_a_Diagonal_Operator}.

\bigskip

{\bf Proof of Theorem \ref{Periodic_Points_of_a_Diagonal_Operator}:}
Theorem \ref{Periodic_Points_of_a_Diagonal_Operator} (1) follows from Proposition \ref{Characterization_of_Periodic_Points}. Theorem \ref{Periodic_Points_of_a_Diagonal_Operator} (2) follows from Proposition \ref{Criterion_for_the_Periodic_Points_to_be_Closed}. Theorem \ref{Periodic_Points_of_a_Diagonal_Operator} (3) and (4) follow from Proposition \ref{Criterion_for_the_Periodic_Points_to_be_the_whole_or_dense}.
$\Box$

We give a short remark about Theorem \ref{Periodic_Points_of_a_Diagonal_Operator} in the finite-dimensional setting.

\begin{remark}
{\rm An analog of Theorem \ref{Periodic_Points_of_a_Diagonal_Operator} can be given in the finite-dimensional setting. Consider $\mathbb{C}^n$ with the standard orthonormal basis $\{e_1, ..., e_n\}.$ Let $T: \mathbb{C}^n \rightarrow \mathbb{C}^n$ be a diagonal linear operator given by $T(e_k) = c_k e_k$ for all $k = 1, ..., n.$ Then we have that $P(T)$ is the linear span of $\{e_k: c_k \in G\}.$ Therefore, we have $P(T) = P(T^*).$ Since $\mathbb{C}^n$ is finite-dimensional, $P(T)$ is clearly closed in $\mathbb{C}^n.$ Also, $P(T)$ is all of $\mathbb{C}^n$ if and only if $\{c_1, ..., c_n\} \subset G$.}
$\Box$
\end{remark}

We record an interesting corollary of Theorem \ref{Periodic_Points_of_a_Diagonal_Operator}.

\begin{proposition}
    Let $T \in \mathcal{B}(H)$ be a compact normal operator. Then $P(T)$ is finite-dimensional.
\end{proposition}
\begin{proof}
    Since $T$ is a compact normal operator, there exists an orthonormal basis $\{u_n\}_{n \in \mathbb{N}}$ for which $T$ is diagonal with respect to $\{u_n\}_{n \in \mathbb{N}}.$ (See Theorem 2.4.4 in p.55 in \cite{Gerard_J_Murphy}.) Assume that $T(u_n) = c_n u_n$ for all $n \in \mathbb{N}.$ By the compactness of $T,$ it follows that $c_n \rightarrow 0$ as $n \rightarrow \infty.$ Hence, there are only finitely many $n$ for which $c_n$ is a root of unity. By Theorem \ref{Periodic_Points_of_a_Diagonal_Operator} (1), it follows that the space $P(T)$ is finite-dimensional. We complete the proof.
\end{proof}

Now we examine some examples. As an easy example of Proposition \ref{Criterion_for_the_Periodic_Points_to_be_the_whole_or_dense}(1), we consider the following. Fix a positive integer $k \in \mathbb{N}.$ For each $n \in \mathbb{N},$ we denote by $n (\mbox{mod } k)$ to be the integer $r \in \{0,1,...,k-1\}$ with the property that $k | (n-r).$ Set $\alpha_n := \exp(2\pi i n (\mbox{mod } k)/k)$ for all $n \in \mathbb{N},$ and let $T \in \mathcal{B}(H)$ be the diagonal operator with $T(e_n) = \alpha_n e_n$ for all $n \in \mathbb{N}.$ Therefore, each $\alpha_n$ is a root of unity and $\sup\{\text{order}(\alpha_n): n \in \mathbb{N}\} < \infty.$ By Proposition \ref{Criterion_for_the_Periodic_Points_to_be_the_whole_or_dense}(1), we have $P(T) = H.$
Now we consider some standard examples of Proposition \ref{Criterion_for_the_Periodic_Points_to_be_the_whole_or_dense}(2).

\begin{example}
\label{standard_example_in_section2}
 {\rm For each $n \in \mathbb{N},$ we let $\xi_n := e^{\frac{2\pi i}{n}} \in \mathbb{C}.$ Let $T \in \mathcal{B}(H)$ be the diagonal operator with $T(e_n) = \xi_n e_n$ for all $n \in \mathbb{N}.$ Hence, each $\xi_n$ is a root of unity, and order$(\xi_n) = n \rightarrow \infty$ as $n \rightarrow \infty.$ By Proposition \ref{Criterion_for_the_Periodic_Points_to_be_the_whole_or_dense}(2), $P(T)$ is a proper dense subspace of $H.$ Here $P(T) = \cup_{n \in \mathbb{N}} \{c_1 e_1+...+c_n e_n: c_1, ..., c_n \in \mathbb{C}\}.$ Hence, $\{e_n : n \in \mathbb{N}\}$ is a Hamel basis of $P(T).$}
$\Box$
\end{example}

We now consider another interesting example.

\begin{example}
\label{interesting_example_in_section2}
{\rm We construct an example of a unitary diagonal operator $T \in \mathcal{B}(H)$ such that $\sigma(T) = S^1 := \{z \in \mathbb{C}: |z| = 1\}$, and such that the space $P(T)$ is a proper dense subspace of $H.$ Let $\{\beta_n\}_{n \in \mathbb{N}}$ be an enumeration of the countable set $G$ so that it is dense in $S^1.$ Let $T \in \mathcal{B}(H)$ be the diagonal operator with $T(e_n) = \beta_n e_n$ for all $n \in \mathbb{N}.$ Then $\sigma(T)$ is the closure of $\{\beta_n\}_{n \in \mathbb{N}}$, which is the whole $S^1.$ Also, it is readily checked that for each $n \in \mathbb{N},$ we have $T^*(e_n) = \bar \beta_n e_n = \beta_n^{-1} e_n = T^{-1}(e_n).$ Therefore, $T$ is a unitary operator. Finally, note that each $\beta_n$ is a root of unity, and $\sup\{\text{order}(\beta_n): n \in \mathbb{N}\} = \infty.$ Thus, by Proposition \ref{Criterion_for_the_Periodic_Points_to_be_the_whole_or_dense}(2), it follows that $P(T)$ is a proper dense subspace of $H.$ As in Example \ref{standard_example_in_section2}, we have that $\{e_n : n \in \mathbb{N}\}$ is a Hamel basis of $P(T).$}
$\Box$
\end{example}

\section{Periodic Points of a Permutation Operator}
In this section, we prove Theorem \ref{Periodic_Points_of_a_Permutation_Operator} and give some related examples. Most of the ideas are similar to those in the first section.

Let $\sigma: \mathbb{N} \rightarrow \mathbb{N}$ be a bijection. It induces an equivalence relation on $\mathbb{N}$,
given by $n$ is equivalent to $m$ if and only if $m = \sigma^k(n)$ for some $k \in \mathbb{Z}.$
For each $m \in \mathbb{N},$ we denote by $[m]$ the equivalence class of $m.$
We define $T : = T_{\sigma} \in \mathcal{B}(H)$ by setting $T_{\sigma}(e_n) = e_{\sigma(n)}$ for all $n \in \mathbb{N}.$
We keep the same meanings for the notations $\sigma$, $T: = T_\sigma$ and $[m]$ throughout this section.
Recall that $P(T)$ denotes the periodic points of $T$.

\begin{theorem}
\label{Characterization_of_Periodic_Points_for_a_Permutation_Operator}
     We have $\cup_{k \in \mathbb{N}} \{\sum_{n = 1}^{\infty} c_n e_n \in H: \text{card}([n]) \leq k \text{ whenever } c_n \neq 0\} \subset P(T)$, and the inclusion can be proper. Also, $P(T) = P(T^*).$
\end{theorem}
\begin{proof}
Let $x = \sum_{n = 1}^{\infty} c_n e_n$, where $\{\text{card}([n]): c_n \neq 0\}$ is bounded above. (Here $\{c_n\}_{n \in \mathbb{N}}$ is a square summable sequence.)
 We pick $M \in \mathbb{N}$ such that card$([n])$ divides $M$ whenever $c_n \neq 0.$ Then
 \[ x = \sum_{n = 1}^{\infty} c_n e_n = \sum_{c_n \neq 0} c_n e_n = \sum_{c_n \neq 0} c_n e_{\sigma^M(n)} = T^M(\sum_{c_n \neq 0} c_n e_n) = T^M(\sum_{n = 1}^{\infty} c_n e_n) = T^M(x). \]
 This implies that $x \in P(T).$

To see that the above inclusion can be proper, we consider an explicit example. Define a bijection $\sigma: \mathbb{N} \rightarrow \mathbb{N}$ by
\[
\begin{array}{l}
 \sigma(1) = 1 \;,\; \sigma(2^k-1) = 2^{k-1} \mbox{ for all } 2 \leq k \in \mathbb{N} \;, \\
\sigma(n) = n+1 \mbox{ for remaining positive integers}.
\end{array}
\]
 That is, $\sigma$ has the cycle decomposition $\sigma = (1)(2 3)(4 5 6 7)(...).$
Consider
\[ x := \sum_{k = 1}^{\infty}\frac{1}{2^{k^2}}\sum_{j = 0}^{2^{k-1}-1}e_{2^k+2j} = \frac{1}{2^1}e_2 + \frac{1}{2^4}(e_4 + e_6) + \frac{1}{2^9}(e_8 + e_{10}+e_{12}+e_{14})+... \in H.\]
 We have $T^2(x) = x,$ so $x \in P(T)$ whereas
 \[ x \notin \cup_{k \in \mathbb{N}} \{\sum_{n = 1}^{\infty} c_n e_n \in H: \text{card}([n]) \leq k \text{ whenever } c_n \neq 0\}.\]
  This shows that the inclusion can be proper in general.

Finally, to see that $P(T) = P(T^*),$ note that $T$ is a unitary operator with $T^* = T^{-1} = T_{\sigma^{-1}}.$ Thus, $P(T) = P(T^{-1}) = P(T^*).$ We complete the proof.
\end{proof}

The following result gives a criterion for when the space $P(T)$ of periodic points is closed in $H.$ The proof of the result parallels Proposition \ref{Criterion_for_the_Periodic_Points_to_be_Closed} in the previous section.
Let card$(X)$ denotes the cardinality of the set $X.$

\begin{proposition}
\label{Criterion_for_the_Periodic_Points_to_be_Closed_for_a_Permutation_Operator}
     We have $P(T)$ is closed in $H$ if and only if $\sup\{\text{card}([m]):m \in \mathbb{N}, [m] \text{ is finite}\} < \infty.$
\end{proposition}
\begin{proof}
Firstly, assume that $\sup\{\text{card}([m]):m \in \mathbb{N}, [m] \text{ is finite}\} = \infty.$ Then we can pick an increasing sequence
of positive integers,
\[ \{n_k\}_{k \in \mathbb{N}} \subset \mathbb{N}  \;,\;
 [n_k] \mbox{ is finite and } \mbox{card}([n_1]) < \mbox{card}([n_2]) < .... \]
  Thus card$([n_k]) \rightarrow \infty$ as $k \rightarrow \infty.$ For each $k \in \mathbb{N},$ we take
  \[ x_k := \sum_{j = 1}^{k} \frac{1}{2^j}e_{n_j} = e_{n_1} + \frac{1}{2}e_{n_2} + ... +\frac{1}{2^k}e_{n_k} \in P(T). \]
  To see that $x_k \in P(T)$, let $M = \prod_{i=1}^k \mbox{card}([n_i])$ and we have
 \[ T^M(x_k) = T^M(\sum_{j = 1}^{k} \frac{1}{2^j}e_{n_j}) = \sum_{j = 1}^{k} \frac{1}{2^j} T^M(e_{n_j}) = \sum_{j = 1}^{k} \frac{1}{2^j} e_{n_j} = x_k .\]

 We claim that
 \begin{equation}
 x := \sum_{j = 1}^{\infty} \frac{1}{2^j}e_{n_j} \notin P(T).
 \label{mee}
 \end{equation}
Assume otherwise, $x \in P(T),$ so there exists $N \in \mathbb{N}$ such that $T^N(x) = x.$ Then
\begin{equation}
 \sum_{j = 1}^{\infty} \frac{1}{2^j}e_{n_j} = x = T^N(x) = T^N(\sum_{j = 1}^{\infty} \frac{1}{2^j}e_{n_j}) = \sum_{j = 1}^{\infty} \frac{1}{2^j} T^N(e_{n_j}) = \sum_{j = 1}^{\infty} \frac{1}{2^j} e_{\sigma^N(n_j)}.
 \label{bop}
 \end{equation}
Compare both sides of (\ref{bop}), and we obtain that $\sigma^N(n_j) = n_j$ for all $j \in \mathbb{N}.$ Therefore, $N \geq $ card$([n_j])$ for all $j \in \mathbb{N}.$ This contradicts the fact that card$([n_k]) \rightarrow \infty$ as $k \rightarrow \infty$,
which proves (\ref{mee}) as claimed.

 We have shown that $x_k \in P(T)$ and
 $x_k \rightarrow x \notin P(T)$, so $P(T)$ is not closed in $H.$

Conversely, assume that $\sup\{\text{card}([m]):m \in \mathbb{N}, [m] \text{ is finite}\} < \infty.$ Then we can pick a smallest positive integer $M \in \mathbb{N}$ with the property that card$([n])$ divides $M$ whenever card$([n])$ is finite.
Let $I$ be the identity operator, and we claim that
\begin{equation}
P(T) = \ker (T^M-I) .
\label{wan}
\end{equation}
We clearly have $\ker(T^M-I) \subset P(T).$ Conversely, assume that $x = \sum_{j = 1}^{\infty} c_j e_j \in P(T).$ (Here $\{c_j\}_{j \in \mathbb{N}}$ is square summable.) Then $T^N(x) = x$ for some $N \in \mathbb{N}.$ Now if $j$ is such that $[j]$ is infinite, then for any $s \in \mathbb{N}$ we have $\sum_{k = 1}^{\infty} c_k e_k = x = T^{sN}(x) = \sum_{k = 1}^{\infty} c_k e_{\sigma^{sN}(k)}.$ This implies that $c_j = c_{\sigma^{sN}(j)}$ for all $s \in \mathbb{N}.$ Combine this with the fact that $[j] = \{\sigma^{k}(j): k \in \mathbb{Z}\}$ is infinite and that $\{c_k\}_{k \in \mathbb{N}}$ is square summable, and we get that $c_j = 0.$ The above argument implies that $[j]$ is finite for all $j$ with $c_j \neq 0.$ Therefore, $x = \sum_{j = 1}^{\infty} c_j e_j = \sum_{[j] \text{ finite}}^{\infty} c_j e_j = \sum_{[j] \text{ finite}}^{\infty} c_j e_{\sigma^{M}(e_j)} = T^M(x).$ This shows that $x \in \text{ker}(T^M-I).$ We have shown that $P(T) \subset \text{ker}(T^M-I)$.
This proves (\ref{wan}) as claimed. Since $\ker (T^M-I)$ is closed, this completes the proof.
\end{proof}

As a corollary of Proposition \ref{Criterion_for_the_Periodic_Points_to_be_Closed_for_a_Permutation_Operator}, we derive the following criterion for determining if the space $P(T)$ of all periodic points of $T \in \mathcal{B}(H)$ is the whole space $H$ or is dense in $H.$
We consider each equivalence class $[m]$ for $m \in \mathbb{N}$.

\begin{proposition}
\label{Criterion_for_the_Periodic_Points_to_be_the_whole_or_dense_for_a_Permutation_Operator}
$\,$
    \\(1) $P(T) = H$ if and only if $\sup\{\text{card}([m]): m \in \mathbb{N}\} < \infty.$
    \\(2) If $\sup\{\text{card}([m]): m \in \mathbb{N}\} = \infty$ and each equivalence class $[m]$ is finite, then $P(T)$ is a proper dense subspace of $H.$
    \\(3) If the closure of $P(T)$ is finite-codimensional in $H,$ then $P(T)$ is dense in $H,$ so in fact the closure of $P(T)$ is the whole space $H.$
\end{proposition}
\begin{proof}
(1) Firstly, assume that $P(T) = H.$ Then for each $n \in \mathbb{N},$ there is a $k \in \mathbb{N}$ such that $T^k(e_n) = e_n,$ so $e_n = e_{\sigma^k(n)},$ which implies that $n = \sigma^k(n).$ This shows that $[m]$ is finite for all $m$. Now note that $P(T) = H,$ which is closed in $H,$ so by Proposition \ref{Criterion_for_the_Periodic_Points_to_be_Closed_for_a_Permutation_Operator}, it follows that $\sup\{\text{card}([m]): m \in \mathbb{N}\} < \infty.$

Conversely, assume that $\sup\{\text{card}([m]): m \in \mathbb{N}\} < \infty.$ We take a positive integer $M \in \mathbb{N}$ such that each $\mbox{card}([m])$ divides $M.$ Then $T^M =I$, whence $P(T) = H.$
\\(2) Assume that $\sup\{\text{card}([m]): m \in \mathbb{N}\} = \infty.$
For each $n \in \mathbb{N}$, we let $k = \mbox{card}([n]) < \infty$ and see that
 $T^k(e_n) = e_{\sigma^k(n)} = e_n$, so $e_n \in P(T)$. Hence $P(T)$ contains the linear span $\cup_{n \in \mathbb{N}} \{c_1 e_1+...+c_n e_n: c_1, ..., c_n \in \mathbb{C}\}$ and is dense in $H.$ However, by (1), we have that $P(T) \neq H.$ So $P(T)$ is a proper dense subspace of $H.$ We complete the proof.
\\(3) We first claim that in this case,
\begin{equation}
\label{each_equivalence_class_is_finite}
[m] \text{ is finite for all } m \in \mathbb{N}.
\end{equation}
To see \eqref{each_equivalence_class_is_finite}, assume by contradiction that there exists $m \in \mathbb{N}$ for which $[m]$ is infinite. Pick $x \in P(T)$ with $T^M(x) = x$ and $x = \sum_{k = 1}^{\infty} c_k e_k$, where $\{c_k\}_{k \in \mathbb{N}}$ is a square-summable sequence.
Let $k \in [m].$ For any $s \in \mathbb{N},$ one has $\sum_{n = 1}^{\infty} c_n e_n = x = T^{sM}(x) = \sum_{n = 1}^{\infty} c_n e_{\sigma^{sM}(n)}.$ So $c_k = c_{\sigma^{sM}(k)}.$ Since $\{c_n\}_{n \in \mathbb{N}}$ is square summable, it implies that $c_k = 0.$
We have shown that $c_k = 0$ whenever $k \in [m]$,
so $x$ is in the closed linear span of $\{e_n : n \notin [m]\}.$ Therefore, $P(T)$ is a subspace of the closed linear span of $\{e_n : n \notin [m]\}.$ This contradicts to the fact that the closed linear span of $\{e_n : n \notin [m]\}$ is of infinite-codimension in $H.$ This proves (\ref{each_equivalence_class_is_finite}) as claimed.

Now, combine \eqref{each_equivalence_class_is_finite} and Theorem \ref{Criterion_for_the_Periodic_Points_to_be_the_whole_or_dense_for_a_Permutation_Operator} (1)(2), and we see that $P(T)$ is dense in $H.$ We complete the proof.
\end{proof}

By combining the above Propositions, we can now prove Theorem  \ref{Periodic_Points_of_a_Permutation_Operator}.
\\{\bf Proof of Theorem  \ref{Periodic_Points_of_a_Permutation_Operator}:}
Theorem  \ref{Periodic_Points_of_a_Permutation_Operator} (1) follows from Proposition \ref{Characterization_of_Periodic_Points_for_a_Permutation_Operator}. Theorem  \ref{Periodic_Points_of_a_Permutation_Operator} (2) follows from Proposition \ref{Criterion_for_the_Periodic_Points_to_be_Closed_for_a_Permutation_Operator}. Theorem  \ref{Periodic_Points_of_a_Permutation_Operator} (3), (4), and (5) follow from Proposition \ref{Criterion_for_the_Periodic_Points_to_be_the_whole_or_dense_for_a_Permutation_Operator}.
$\Box$
\\ Similarly as in the case of diagonal operators, we briefly discuss Theorem \ref{Periodic_Points_of_a_Permutation_Operator} in the finite-dimensional setting.

\begin{remark}
   {\rm As in the previous section, an analog of Theorem \ref{Periodic_Points_of_a_Permutation_Operator} can be formulated in the finite dimensional setting. Consider $\mathbb{C}^n$ with the standard orthonormal basis $\{e_1, ..., e_n\}.$ Let $\sigma: \{1,...,n\} \rightarrow \{1,...,n\}$ be a bijection. Define $T := T_{\sigma}: \mathbb{C}^n \rightarrow \mathbb{C}^n$ by $T_{\sigma}(e_k) = e_{\sigma(k)}$ for all $k = 1, ..., n.$ Then clearly $P(T) = P(T^*) = \mathbb{C}^n.$}
$\Box$
\end{remark}

\section{A Characterization of Unitary Diagonal Operators}
In this section, we prove Theorem \ref{Apprximation_by_Diagonal_Operators_with_allspace_as_Fixed_Points} and give some examples. Note that if $T \in \mathcal{B}(H)$ is a diagonal operator with $T(e_n) = \alpha_n e_n$ for all $n \in \mathbb{N}$ and $P(T) = H,$ then by Theorem \ref{Periodic_Points_of_a_Diagonal_Operator}(4), it follows that $|\alpha_n| = 1$ for all $n \in \mathbb{N},$ and so $T$ is unitary. Now by the fact that the set of all unitary diagonal operators on $H$ is closed in $\mathcal{B}(H)$ (with respect to the norm topology on $\mathcal{B}(H)$), it follows that the limit of a sequence $\{T_n\}_{n \in \mathbb{N}}$ (where $P(T_n) = H$ for all $n$) must be a unitary diagonal operator. Theorem \ref{Apprximation_by_Diagonal_Operators_with_allspace_as_Fixed_Points} shows that this property in fact characterizes the unitary diagonal operators. That is, every unitary diagonal operator can be approximated by a sequence $\{T_n\}_{n \in \mathbb{N}}$ of diagonal operators with $P(T_n) = H$ for all $n.$ The theorem also shows that the set of all diagonal operators $T$ with $P(T) = H$ is dense in the set of all unitary diagonal operators.

\begin{remark}
{\rm In general it is not true that a unitary operator $S \in \mathcal{B}(H)$ is diagonal with respect to some orthonormal basis. As an example, consider the case where $H = l^2(\mathbb{Z}).$ Take $S$ to be the bilateral shift operator: $S(\delta_k) = \delta_{k+1}$ for all $k \in \mathbb{Z}.$ (Here $\delta_k \in l^2(\mathbb{Z})$ is defined by $\delta_k(n) = 1$ if $n = k$ and $\delta(n) = 0$ if $n \neq k.$) The bilateral shift operator $S$ is unitary. However, $S$ has no eigenvectors, so $S$ is not diagonal with respect to any orthonormal basis.}
$\Box$
\end{remark}

{\bf Proof of Theorem \ref{Apprximation_by_Diagonal_Operators_with_allspace_as_Fixed_Points}:}

Firstly, we show that (2) implies (1). Assume that there exists a sequence $\{T_n\}_{n \in \mathbb{N}}$ of diagonal operators such that $P(T_n) = H$ for all $n \in \mathbb{N}$ and $\norm{T_n - T} \rightarrow 0$ as $n \rightarrow \infty.$ By Theorem \ref{Periodic_Points_of_a_Diagonal_Operator}(4), it follows that each $T_n$ is a unitary operator. By the fact that the set of all unitary operators is closed in $\mathcal{B}(H)$ (under the usual norm topology in $\mathcal{B}(H)$), it follows that $T$ is unitary, whence (1).

Conversely, we show that (1) implies (2). Let $T \in \mathcal{B}(H)$ be a unitary operator with $T(e_n) = \alpha_n e_n$ for all $n \in \mathbb{N}.$
Take $E_n := \{e^{\frac{2 \pi i k}{2^n}}: k \in \mathbb{Z}\}$ for all $n \in \mathbb{N}.$ Hence, $E_n$ is the finite cyclic subgroup of $S^1$ generated by $e^{\frac{2 \pi i}{2^n}},$ and $E_n$ has cardinality $2^n.$
Let $\beta_{j,n} \in E_n$ be the element such that $|\beta_{j,n}-\alpha_j| = \min \{|z-\alpha_j|: z \in E_n\}.$
We construct $T_n$ by $T_n(e_j) = \beta_{j,n}e_j.$
Thus, each $T_n$ is a unitary diagonal operator with $T_n^{2^n} = I.$ Therefore, $P(T_n) = H$ for all $n \in \mathbb{N}.$ Moreover, $\norm{T_n-T} = \sup_{j \in \mathbb{N}}|\beta_{j,n}-\alpha_j| \leq \frac{2\pi}{2^n} \rightarrow 0$ as $n \rightarrow \infty.$ Therefore, (2) holds.
$\Box$

Note that one cannot give an analog of Theorem \ref{Apprximation_by_Diagonal_Operators_with_allspace_as_Fixed_Points} for permutation operators, as explained in the following remark.

\begin{remark}
    {\rm Assume that $\sigma: \mathbb{N} \rightarrow \mathbb{N}$ and $\tau:  \mathbb{N} \rightarrow \mathbb{N}$ are both bijections, and $\sigma \neq \tau.$ Hence $\sigma(i) \neq \tau(i)$ for some $i \in \mathbb{N}.$ It follows that $T_{\sigma}(e_i) = e_{\sigma(i)}$ is of distance $\sqrt{2}$ from $T_{\tau}(e_i) = e_{\tau(i)}.$ This implies that $\norm{T_{\sigma}-T_{\tau}} \geq \sqrt{2}$ whenever $\sigma$ and $\tau$ are distinct.
    Consequently, given an arbitrary bijection $\sigma$ on $\mathbb{N},$ in general one cannot find a sequence $\{\sigma_n\}_{n \in \mathbb{N}}$ of bijections on $\mathbb{N}$ with the property that $P(T_{\sigma_n}) = H$ for all $n \in \mathbb{N}$ and $\norm{T_{\sigma}-T_{\sigma_n}} \rightarrow 0$ as $n \rightarrow \infty.$}
$\Box$
\end{remark}
As an interesting corollary of the above Remark, we give another short proof of the following well-known result. Recall that for any $C^*-$algebra $A$ with identity, we denote by $U(A) := \{a \in A: a a^* = a^* a = 1\}$ the group of all unitary elements of $A.$

\begin{proposition}
(1) The unit sphere $\{T \in \mathcal{B}(H): \norm{T} = 1\}$ of $\mathcal{B}(H)$ is not separable under the operator norm topology.
\\(2) The unitary group $U(\mathcal{B}(H)) = \{T \in \mathcal{B}(H): T T^* = T^* T = I\}$ of $\mathcal{B}(H)$ is not-separable and noncompact under the operator norm topology of $\mathcal{B}(H)$.
\end{proposition}
\begin{proof}
(1) We take $S_{\mathbb{N}}$ to be the group of all bijections of the set $\mathbb{N}$ onto itself. Note that $S_{\mathbb{N}}$ is uncountable. Hence, $\{T_{\sigma}\}_{\sigma \in S_{\mathbb{N}}}$ is an uncountable subset of the unit sphere $\{T \in \mathcal{B}(H): \norm{T} = 1\}$ with $\norm{T_{\sigma}-T_{\tau}} \geq \sqrt{2}$ whenever $\sigma$ and $\tau$ are distinct bijections on $\mathbb{N}.$ This shows that $\{T \in \mathcal{B}(H): \norm{T} = 1\}$ is not separable.
\\(2) As in the proof of part (1), $\{T_{\sigma}\}_{\sigma \in S_{\mathbb{N}}}$ is an uncountable subset of the unitary group $U(\mathcal{B}(H)) = \{T \in \mathcal{B}(H): T T^* = T^* T = I\}$ of $\mathcal{B}(H)$ with $\norm{T_{\sigma}-T_{\tau}} \geq \sqrt{2}$ whenever $\sigma$ and $\tau$ are distinct. This shows that $U(\mathcal{B}(H)) = \{T \in \mathcal{B}(H): T T^* = T^* T = I\}$ is not separable, hence noncompact under the operator norm topology on $\mathcal{B}(H).$ We complete the proof.
\end{proof}

Recall that if we take $H$ to be $\mathbb{C}^n,$ then the unitary group $U(\mathcal{B}({\mathbb{C}^n}))$ is the familiar group $U(n,\mathbb{C})$ of all $n$ by $n$ unitary complex matrices, which is a compact Lie group. Consequently, for a separable Hilbert space $H,$ the unitary group $U(\mathcal{B}(H))$ is compact under the operator norm topology if and only if $H$ is finite-dimensional.

{}
\end{document}